\documentclass[12pt,a4paper]{article}
\usepackage[utf8]{inputenc}
\usepackage{graphicx}
\usepackage{amsmath}
\usepackage{amssymb}
\usepackage{amscd}
\usepackage{amsthm}
\usepackage{cleveref}

\newtheorem{theorem}{Theorem}[section]
\newtheorem{corollary}{Corollary}
\newtheorem{lemma}[theorem]{Lemma}
\newtheorem{proposition}{Proposition}

\newtheorem{definition}[theorem]{Definition}
\newtheorem{remark}{Remark}
\newtheorem{example}{Example}
\newcommand{\EQ}{\begin{equation}}
\newcommand{\EN}{\end{equation}}
\newcommand{\w}{\mbox{wt}}

\newcommand{\zero}{{\mathbf{0}}}
\newcommand{\one}{{\mathbf{1}}}

\newcommand{\bh}{{\bf h}}
\newcommand{\bc}{{\bf c}}

\newcommand{\by}{{\bf y}}
\newcommand{\bx}{{\bf x}}
\newcommand{\bz}{{\bf z}}
\newcommand{\bv}{{\bf v}}

\newcommand{\bw}{{\bf w}}

\newcommand{\F}{\mathbb{F}}

\newcommand{\IA}{\operatorname{IA}}

\newcommand{\Aut}{\operatorname{Aut}}
\newcommand{\MAut}{\operatorname{MAut}}

\newcommand{\GL}{\operatorname{GL}}

\title{On completely regular and completely transitive codes derived from Hamming codes}

\author{J. Borges, J. Rif\`{a} and V. A. Zinoviev}

\begin{document}
\maketitle

\centerline{\scshape Joaquim Borges}
\medskip
{\footnotesize
 \centerline{Department of Information and Communications Engineering}
 \centerline{Universitat Aut\`{o}noma de Barcelona}
} 

\medskip

\centerline{\scshape Josep Rifà}
\medskip
{\footnotesize
 \centerline{Department of Information and Communications Engineering}
 \centerline{Universitat Aut\`{o}noma de Barcelona}
} 

\centerline{\scshape Victor Zinoviev}
\medskip
{\footnotesize
 \centerline{A.A. Kharkevich Institute for Problems of Information Transmission}
 \centerline{Russian Academy of Sciences}
} 

\bigskip

\begin{abstract}
Given a parity-check matrix $H_m$ of a $q$-ary Hamming code, we consider a partition of the columns into two subsets. Then, we consider the two codes that have these submatrices as parity-check matrices. We say that anyone of these two codes is the supplementary code of the other one.

We obtain that if one of these codes is a Hamming code, then the supplementary code is completely regular and completely transitive. If one of the codes is completely regular with covering radius $2$, then the supplementary code is also completely regular with covering radius at most $2$. Moreover, in this case, either both codes are completely transitive, or both are not.

With this technique, we obtain infinite families of completely regular and completely transitive codes which are quasi-perfect uniformly packed. 
\end{abstract}

\section{Introduction}
Let $\F_q$ be the finite field of order $q$. The {\em weight} of a vector $\bv\in\F_q^n$, denoted by $\w(\bv)$, is the number of nonzero coordinates of $\bv$. The vector of weight $0$, or {\em zero vector}, is denoted by $\zero$. The {\em distance} between two vectors $\bv,\bw\in\F_q^n$, denoted by $d(\bv,\bw)$, is the number of coordinates in which they differ. A subset $C\subset\F_q^n$ is called a $q$-{\em ary code} of length $n$. Denote by $d$ the {\em minimum distance} among codewords in $C$.  The {\em packing radius} of $C$ is $e=\lfloor (d-1)/2 \rfloor$ and $C$ is said to be an $e$-{\em error-correcting} code. Given any vector $\bv \in \F_q^n$, its
{\em distance to the code $C$} is $d(\bv,C)=\min_{\bx \in C}\{
d(\bv, \bx)\}$ and the {\em covering radius} of the code $C$ is
$\rho=\max_{\bv \in \F_q^n} \{d(\bv, C)\}$. Note that $e\leq \rho$. If $e=\rho$, then $C$ is a {\em perfect} code. If $e=\rho-1$, then $C$ is called a {\em quasi-perfect} code. If $C$ is a $k$-dimensional subspace of $\F_q^n$, then $C$ is {\em linear} and referred to as an $[n,k,d;\rho]_q$-code. If $C$ is linear of length $n$ and dimension $k$, then a {\em generator matrix} $G$ for $C$ is any $k\times n$ matrix with $k$ linearly independent codewords as rows. A {\em parity-check matrix} for $C$ is an $(n-k)\times n$ matrix $H$ such that $C$ is the null space of $H$, i.e. $H\bx^T=\zero^T$ if and only if $\bx\in C$. The {\em dual code} $C^\perp$ is the orthogonal subspace to $C$. Hence, $H$ generates $C^\perp$ and $G$ is a parity-check matrix for $C^\perp$.

A linear single-error-correcting ($e=1$) perfect code is called a {\em Hamming} code. Such a code has parameters
$$
[n=(q^m-1)/(q-1), k=n-m, d=3;\rho=1]_q\;\;\;(m>1)
$$
and is denoted by ${\cal H}_m$. A parity-check matrix for ${\cal H}_m$, denoted by $H_m$, contains a maximal set of $n=(q^m-1)/(q-1)$ pairwise linearly independent column vectors of length $m$ \cite{MacW}. The dual code ${\cal H}_m^\perp$ generated by $H_m$ is called {\em simplex} and it is a constant-weight code, that is, all nonzero codewords have the same weight $q^{m-1}$.

We denote by $~D=C+\bx~$ a
{\em coset} of  $C$, where $+$ means the componentwise addition
in $\F_q$.

For a given $q$-ary code $C$ of length $n$ and covering radius $\rho$,
define
\[
C(i)~=~\{\bx \in \F_q^n:\;d(\bx,C)=i\},\;\;i=0,1,\ldots,\rho.
\]
The sets $C(0)=C,C(1),\ldots,C(\rho)$ are called the {\em subconstituents} of $C$.

Say that two vectors $\bx$ and $\by$ are {\em neighbors} if
$d(\bx,\by)=1$. Given two vectors $\bx=(x_1,\ldots,x_n),\by=(y_1,\ldots,y_n)\in\F_q^n$, we say that $\by$ {\em covers} $\bx$ if $y_i=x_i$, for all $i$ such that $x_i\neq 0$.
\begin{definition}[\cite{Neum}]\label{de:1.1} A $q$-ary code $C$ of length $n$ and covering radius $\rho$ is {\em completely regular}, if
for all $l\geq 0$ every vector $\bx \in C(l)$ has the same number
$c_l$ of neighbors in $C(l-1)$ and the same number $b_l$ of
neighbors in $C(l+1)$. Define $a_l = (q-1){\cdot}n-b_l-c_l$ and
set $c_0=b_\rho=0$. The parameters $a_l$, $b_l$ and $c_l$ ($0\leq l\leq \rho$) are called {\em intersection numbers} and the sequence $\IA=\{b_0, \ldots, b_{\rho-1}; c_1,\ldots,
c_{\rho}\}$ is called the {\em intersection array} of $C$.
\end{definition}

Let $M$ be a monomial matrix, i.e. a matrix with exactly one
nonzero entry in each row and column. Such a matrix can be written as $M=DP$, where $D$ is a monomial diagonal matrix and $P$ is permutation matrix. If $q$ is prime, then
the automorphism group of $C$, $\Aut(C)$, consists of all monomial $n\times n$ matrices $M$ over
$\F_q$ such that $\bx M \in C$ for all $\bx \in C$. If $q$ is a
power of a prime number, then the monomial automorphism group of $C$ is denoted by $\MAut(C)$, however, $\Aut(C)$ also contains any field
automorphism of $\F_q$ which preserves $C$.

\begin{lemma}\label{transdual}
If $DP$ is the corresponding matrix to an automorphism $\alpha$ of a code (where $D$ is a monomial diagonal matrix and $P$ is a permutation matrix), then $D^{-1}P$ corresponds to an automorphism $\alpha'$ of the dual code.
\end{lemma}

\begin{proof}
See \cite[Thm. 1.7.9, p. 27]{HP}.
\end{proof}

\begin{remark}\label{transitius}
As a consequence of Lemma \ref{transdual}, $\alpha$ and $\alpha'$ are both transitive on the set of one-weight vectors, or both are not. Note also that if, for a code $C$, $\MAut(C)$ is transitive, then so is $\Aut(C)$ since $\MAut(C)\subseteq\Aut(C)$.
\end{remark}

It is well known, e.g. see \cite{MacW}, that the monomial automorphism group of a Hamming code ${\cal H}_m$ is isomorphic to the general linear group $\GL(m,q)$, which acts transitively on the set of one-weight vectors. In the binary case, the action of $\GL(m,2)$ on the set of coordinate positions is even doubly transitive.

The group $\Aut(C)$
acts on the set of cosets of $C$ in the following way: for all
$\pi\in \Aut(C)$ and for every vector $\bv \in \F_q^n$ we have
$\pi(\bv + C) = \pi(\bv) + C$.

\begin{definition}[\cite{Giu,Sole}]\label{de:1.3}
Let $C$ be a linear code over $\F_q$ with covering radius $\rho$.
Then $C$ is {\em completely transitive} if $\Aut(C)$ has $\rho +1$ orbits
when acts on the cosets of $C$.
\end{definition}

Since two cosets in the same orbit have the same weight
distribution, it is clear that any completely transitive code is
completely regular.

Completely regular and completely transitive codes are classical
subjects in algebraic coding theory, which are closely connected
with graph theory, combinatorial designs and algebraic
combinatorics. Existence, construction and enumeration of all such
codes are open hard problems (see \cite{BRZ2,BRZ3,BCN,Koo,Neum,Dam} and
references there).

It is well known that new completely regular codes can be obtained by the direct sum of perfect codes or,
more general, by the direct sum of completely regular codes with covering radius $1$ \cite{BZZ,Sole}.

In the current paper, starting from Hamming codes and choosing appropriate columns of their parity-check matrix, we obtain parity-check matrices for completely regular codes. More precisely, given the parity-check matrix $H_m$ of a $q$-ary Hamming code, we consider a partition of the columns of $H_m$ into two subsets. We consider these two subsets of columns as parity-check matrices of two codes, $A$ and $B$. We say that $B$ is the {\em supplementary code} of $A$ (and $A$ is the supplementary code of $B$). If $A$ or $B$ is a Hamming code, then the supplementary code is also completely regular and completely transitive. We point out that, in this case, the dual code of the supplementary code belongs to the family SU1 in \cite{Cald}. If $A$ or $B$ is a completely regular code with covering radius $2$, then the supplementary code is completely regular with covering radius at most $2$. Moreover, in this situation both codes are completely transitive or are not, simultaneously.

In this way, we construct infinite families of $q$-ary completely regular and completely transitive codes. It is worth mentioning that for fixed $q$, we obtain a growing number of completely regular codes as the length of the starting Hamming code increases.

In the next section, we recall several known results on completely regular codes, which we shall use later. The main results and constructions are presented in \Cref{main}.

\section{Preliminary results}\label{preliminars}

In this section we see several results we will need in the next sections.

\begin{lemma}[\protect{\cite{Neum}}]\label{graph}
Let $C$ be a completely regular code with covering radius $\rho$ and intersection array $\{b_0,\ldots,b_{\rho-1};c_1,\ldots,c_\rho\}$.
If $C(i)$ and $C(i+1)$, $0\leq i <\rho$, are two subconstituents of $C$, then
$$
b_i|C(i)|=c_{i+1}|C(i+1)|.
$$
\end{lemma}

Let $C\subset\F_q^n$ be a code. For any vector $\bx\in\F_q^n$ and for all $j=0,\ldots,n$, define $B_{\bx,j}$ as the number of codewords at distance $j$ from $\bx$:
$$
B_{\bx,j}=|\{\bz\in C\mid d(\bx,\bz)=j\}|.
$$

\begin{definition}[\protect{\cite{GvT}}]\label{def:2.5}
A quasi-perfect $e$-error-correcting $q$-ary code $C$ is called
{\em uniformly packed} if there exist natural numbers $\lambda$ and $\mu$
such that for any vector $\bx$:
$$
B_{\bx,e+1} = \left\{
\begin{array}{cl}
  \lambda & \mbox{ if } d(\bx,C)=e, \\
  \mu     & \mbox{ if } d(\bx,C)=e+1.
\end{array}\right.
$$
\end{definition}

Van Tilborg \cite{vTi} (see also \cite{Lind,SZZ})
showed that no nontrivial codes of this kind exist for
$e>3$.

\begin{proposition}[\protect{\cite{GvT}}, see also \cite{SZZ}]\label{UPCR}
A uniformly packed code is completely regular.
\end{proposition}

For a code $C$, we denote by $s+1$ the number of nonzero terms in the dual
distance distribution of $C$, obtained by the MacWilliams transform. The parameter
$s$ was called {\em external distance} by Delsarte \cite{Del}, and is equal
to the number of nonzero weights of $C^\perp$ if $C$ is linear. The following properties show the importance of this parameter.

\begin{theorem}\label{params}
If $C$ is any code with packing radius $e$, covering radius $\rho$, and external distance $s$, then
\begin{itemize}
\item[(i)] {\cite{Del}} $\rho \leq s$.
\item[(ii)] {\cite{Del}} $C$ is perfect ($e=\rho$) if and only if $e=s$.
\item[(iii)] {\cite{GvT}} $C$ is quasi-perfect uniformly packed if and only if $s=e+1$.
\item[(iv)] {\cite{Sole}} If $C$ is completely regular, then $\rho=s$.
\end{itemize}
\end{theorem}

\section{The new construction of completely regular codes}\label{main}

Let $H_m$ be the parity-check matrix of a $q$-ary Hamming code ${\cal H}_m$ of length $n=(q^m-1)/(q-1)$, where $m>1$. Take a non-empty subset of $n_A<n$ columns of $H_m$ as the parity-check matrix of a code $A$. Call $B$ the supplementary code that has as parity-check matrix the remaining $n_B=n-n_A$ columns of $H_m$. In this section, we see that if $A$ or $B$ is a completely regular code with covering radius $\rho(A)\leq 2$, then so is the supplementary code, under certain conditions.

For the rest of this section, we write $n_j=(q^j-1)/(q-1)$, for any integer value $j>0$.

\subsection{The case $\rho(A)=1$}\label{subseccio1}

Since there are no two linearly dependent columns in $H_m$, we have that, for $n_A\geq 3$, the minimum distance of $A$ (and of $B$, for $n_A\leq n-3$) is at least three and thus, the packing radius is at least 1. If $\rho(A)=1$ (hence $n_A\neq 2$), then $e=\rho(A)=1$ for $n_A\geq 3$, by \Cref{params}. Therefore, $A$ is a perfect Hamming code for $n_A>1$.

For $u\in\{1,\ldots,m-1\}$, $H_m$ can be written as:

\begin{equation}\label{form1}
H_m=\left[
\begin{array}{c|c}
H^*_u & H_{u,m}
\end{array}
\right],
\end{equation}

\noindent where the first $u$ rows of $H^*_u$ are as the parity-check matrix of ${\cal H}_u$ and the remaining $m-u$ rows are all-zero vectors. For the case $u=1$, the matrix $H^*_u$ is simply the column vector $(1,0,\ldots,0)^T$. We call $B_{u,m}=B$ the code that has parity-check matrix $H_{u,m}$. Note that for $n_A>1$, we have $A={\cal H}_u$.

\begin{lemma}\label{pesos}
The dual code of $B_{u,m}$, i.e. the code $B_{u,m}^\perp$ generated by $H_{u,m}$ has exactly two nonzero weights, namely, $w_1=q^{m-1}$ and $w_2=q^{m-1}-q^{u-1}$.
\end{lemma}

\begin{proof}
Clearly, $H^*_u$ generates the simplex code, i.e. the dual of the Hamming code, of length $n_u=(q^u-1)/(q-1)$. Hence any vector generated by $H^*_u$ has weight $0$ or $q^{u-1}$. Since any nonzero vector generated by $H_m$ has weight $q^{m-1}$, the result follows.
\end{proof}

\begin{proposition}\label{parB}
The code $B_{u,m}$ has parameters
$$
[n_B=(q^m-q^u)/(q-1),k=(q^m-q^u)/(q-1)-m,d;\rho=2]_q,\;\;\mbox{where }
$$
$$
d=\left\{\begin{array}{cl}
4 & \mbox{if } u=m-1,q=2;\\
3 & \mbox{otherwise.}
\end{array}\right.
$$
\end{proposition}

\begin{proof}
The length $n_B$ of $B_{u,m}$ is simply the length of ${\cal H}_m$ minus the number of columns of $H^*_u$. The dimension $k$ is the length of $B_{u,m}$ minus the number of rows of $H_{u,m}$ (or $H_m$).

Of course, $H_{u,m}$ has no scalar multiple columns, hence $d>2$. Given two columns $\bh_i$ and $\bh_j$ of $H_{u,m}$ we know that there is a column $\bh_\ell$ in $H_m$ which is linearly dependent with $\bh_i$ and $\bh_j$. If $u<m-1$ or $q>2$, we can choose $\bh_i$ and $\bh_j$ such that the last $m-u$ entries are linearly independent, then $\bh_\ell$ cannot be one of the first $n_u$ columns of $H_m$. Indeed, those columns have zeros in the last $m-u$ entries. Hence, $B_{u,m}$ contains codewords of weight $3$. For the case $u=m-1$ and $q=2$, the previous argument does not work since the last row of $H_{u,m}$ is the all-ones vector. Thus, $H_m$ can be written as:

\begin{equation}\label{form2}
H_m=\left[
\begin{array}{c|c|c}
H_{u} & H_{u} & \zero^T\\
\zero & \one & 1
\end{array}
\right].
\end{equation}
\
In fact, in this case, $B_{u,m}$ is the binary extended Hamming code of length $2^u$ and, therefore, it has minimum weight $4$.

Finally, since $B_{u,m}$ is not perfect, $\rho > e=1$ and, by Lemma \ref{pesos}, $B_{u,m}$ has external distance $s=2$, hence $\rho\leq 2$ by Theorem \ref{params}.
\end{proof}

\begin{lemma}\label{comptes}
The number of vectors at distance 1 and at distance 2 from $B_{u,m}$ are, respectively:
\begin{eqnarray*}
|B_{u,m}(1)| &=& q^{n_B-m}(q^m-q^u), \mbox{ and }\\
|B_{u,m}(2)| &=& q^{n_B-m}(q^u-1),
\end{eqnarray*}
where $n_B=(q^m - q^u)/(q-1)$ is the length of $B_{u,m}$.
\end{lemma}

\begin{proof}
The number of vectors of weight 1 is $(q-1)n_B$. All these vectors are at distance 1 from exactly one codeword (the zero vector). Thus, $|B_{u,m}(1)|=(q-1)n_B|B_{u,m}|=q^{n_B-m}(q^m-q^u)$.

Since the covering radius of $B_{u,m}$ is $\rho=2$, we have that
\begin{equation*}
\begin{split}
    |B_{u,m}(2)|=&|\F_q^{n_B}|-|B_{u,m}(1)|-|B_{u,m}|=\\&q^{n_B}-q^{n_B-m}(q^u-1)-q^{n_B-m}=q^{n_B-m}(q^u-1).
\end{split}
\end{equation*}
\end{proof}

\begin{corollary}
The code $B_{u,m}$ is quasi-perfect uniformly packed (hence completely regular) with intersection array:
$$
\IA=\{q^m-q^u, q^u-1;1,q^m-q^u\}.
$$
\end{corollary}

\begin{proof}
Since $s=\rho=e+1$, $B_{u,m}$ is a quasi-perfect uniformly packed code, by Theorem \ref{params}. Since $d\geq 3$, it is clear that $b_0=(q-1)n_B=q^m-q^u$ and $c_1=1$. Given a vector $\bx$ of weight 1, the vectors $\by$ of weight 2 covering $\bx$ not at distance one from $B_{u,m}$ are those which are covered by codewords of ${\cal H}_m$ of weight 3, but not in $B_{u,m}$, hence with the third nonzero coordinate in the first $n_u$ positions. In other words, for $\bx$ we can choose anyone of these $n_u$ first positions and, for each of these positions, anyone of the $q-1$ multiples. Therefore $\bx$ is covered by $(q-1)n_u=q^u-1$ vectors of weight 2 at distance 2 from $B_{u,m}$. Thus, we obtain $b_1=q^u-1$.

By Lemma \ref{graph}, we know that $b_1|B_{u,m}(1)|=c_2|B_{u,m}(2)|$. Applying Lemma \ref{comptes}, we obtain:
$$
c_2=\frac{(q^u-1)q^{n_B-m}(q^m-q^u)}{q^{n_B-m}(q^u-1)}=q^m-q^u.
$$
\end{proof}

\begin{remark}
It is not difficult to prove directly that given a vector $\bx\in B_{u,m}(2)$, any neighbor of $\bx$ must be in $B_{u,m}(1)$, obtaining the value of $c_2$.
\end{remark}

Denote by $(\bx\mid \bx')=(x_1,\ldots,x_{n_u}\mid x'_{n_u+1},\ldots,x'_{n_m})$ a vector in $\F_q^{n_m}$ such that $\bx\in\F_q^{n_u}$ and $\bx'\in\F_q^{n_m-n_u}$. Let $e_j$ denote any one-weight vector with its nonzero coordinate at position $j$.

\begin{lemma}\label{cosets}
The number of cosets of $B_{u,m}$ of minimum weight $2$ is $q^u-1$. Moreover, for any vector $\bx'\in\F_q^{n_m-n_u}$ in one such coset, the vector $(\zero\mid \bx')$ is contained in a coset of weight $1$ of ${\cal H}_m$ with leader $e_j$, which has its nonzero coordinate at position $j\in\{1,\ldots,n_u\}$.
\end{lemma}

\begin{proof}
The total number of cosets of $B_{u,m}$ is $q^{n_B}/q^{n_B-m}=q^m$. Since there are one coset of minimum weight 0 (the code $B_{u,m}$) and $(q-1)n_B=q^m-q^u$ cosets of minimum weight 1, we obtain that the number of cosets of minimum weight 2 is $q^m-(q^m-q^u)-1=q^u-1$.

Since $d(\bx',B_{u,m})=2$, we have that there is some codeword $\bc'\in B_{u,m}$ such that $\by'=\bx'-\bc'$ has weight $2$. Hence, $(\zero\mid y')$ is covered by some codeword (of weight $3$) $(e_j\mid \by')\in {\cal H}_m$. Thus, $(\zero\mid \by')=(\zero\mid \bx' - \bc')\in {\cal H}_m-e_j$. Note that $(\zero\mid \bc')\in {\cal H}_m$. Then, $(\zero\mid \bx' - \bc') + (\zero\mid c')\in {\cal H}_m - e_j$, implying $(\zero\mid \bx')\in {\cal H}_m - e_j$.
\end{proof}

%

The matrix $H_m$ (\ref{form1}) can be written as:
\begin{equation}\label{form3}
H_m=\left[
\begin{array}{c|c|c|c|c}
H_u & H_u & \cdots & H_u & \zero_{u,n_{m-u}} \\
\hline
\zero_{m-u,n_u} & G_1 & \cdots & G_{q^{m-u}-1} & H_{m-u}
\end{array}
\right],
\end{equation}
\noindent where $\zero_{i,j}$ stands for the all-zero matrix of size $i\times j$ and $G_1,\ldots,G_{q^{m-u}-1}$ are $m-u \times n_u$ matrices, each one with identical nonzero columns and such that no two columns of distinct $G_i's$ are equal. To see that the matrix (\ref{form3}) is equivalent to the matrix (\ref{form1}), note that no two columns of the matrix (\ref{form3}) are linearly dependent. Therefore, the matrix (\ref{form3}) is a parity-check matrix for ${\cal H}_m$. Indeed the total number of columns is $q^{m-u}n_u + n_{m_u}=n_m$.

For $i=0,\ldots,q^{m-u}$, we call $i$-{\em block} of coordinate positions the set $\{in_u+1,\ldots,(i+1)n_u\}$. Thus, the first block, or $0$-block, corresponds to $\{1,\ldots,n_u\}$. For $i=1,\ldots,q^{m-u}-1$, the $i$-block corresponds to the set of coordinates of the matrix $G_i$. Finally, the last block, or $q^{m-u}$-block, corresponds to the coordinates of the matrix $H_{m-u}$.

\begin{lemma}\label{autos}
If $\alpha\in\Aut({\cal H}_u)$ (acting on the coordinates $\{1,\ldots,n_u\}$), then there exists $\beta\in\Aut(B_{u,m})$ (acting on the coordinates $\{n_u+1,\ldots,n_m\}$) such that $\gamma=(\alpha\mid \beta)\in \Aut({\cal H}_m)$.
\end{lemma}

\begin{proof}
Given $\alpha\in\Aut({\cal H}_u^\perp)$, consider $\gamma=(\alpha\mid\alpha_1\mid\cdots\mid\alpha_{q^{m-u}-1}\mid id)$, where the action of each $\alpha_i$ is identical to the action of $\alpha$ but on the corresponding $i$-block of coordinate positions, and $id$ is the identity on the last block of coordinates. Clearly, $\gamma\in\Aut({\cal H}_m^\perp)$ and $\beta=(\alpha_1\mid\cdots\mid\alpha_{q^{m-u}-1}\mid id)\in\Aut(B_{u,m}^\perp)$.
By Lemma \ref{transdual}, the result follows.
\end{proof}

\begin{proposition}\label{transitiu}
The automorphism group $\Aut(B_{u,m})$ is transitive (on the set of one-weight vectors with coordinates in $\{n_u+1,\ldots,n_m\}$).
\end{proposition}

\begin{proof}
Recall that the automorphism group of a Hamming code ${\cal H}_m$ is isomorphic to $\GL(m,q)$, which acts transitively on the set of one-weight vectors.

Consider the parity-check matrix of ${\cal H}_m$ given in $(\ref{form2})$.
Consider the $m\times m$ matrices $H_{K,M,N}$, where $K,M$ are $u\times u$, $(m-u)\times (m-u)$, nonsingular matrices, respectively, and $N$ is a $u\times (m-u)$ matrix.
$$
H_{K,M,N} = \begin{pmatrix}
K & N\\
0 & M
\end{pmatrix}.
$$

The matrices $H_{K,M,N}$ are in $\GL(m,q)$ and act on $H_m$ as monomial automorphisms, stabilising the Hamming code ${\cal H}_u$, so we can consider these matrices as automorphisms of $B_{u,m}$. Now, we want to show that these matrices assure the transitivity of $\Aut(B_{u,m})$. Take the $i$th and $j$th columns, say $\bh_i$ and $\bh_j$, respectively, where $i,j\in\{n_u+1,\ldots,n_m\}$. We want to find appropriate matrices $K,M,N$ such that $H_{K,M,N}(\bh_i)=\lambda\bh_j$, for any $\lambda\in\F_q$.

Take the projections of both $\bh_i$, $\bh_j$ on the first $u$ coordinates, say $\bh_i^{(u)}$ and $\bh_j^{(u)}$, respectively. And also let $\bh_i^{(m-u)}$ and $\bh_j^{(m-u)}$ be the respective projections on the last $m-u$ coordinates.

First of all, consider the case when $i$ and $j$ are not in the last block of coordinate positions, so that $\bh_i^{(u)}$ and $\bh_j^{(u)}$ are nonzero vectors. Now, take $N=0$, take the matrix $K$ such that $K(\bh_i^{(u)})=\lambda\bh_j^{(u)}$ and the matrix $M$ such that $M(\bh_i^{(m-u)})=\lambda\bh_j^{(m-u)}$. Indeed, we can do these last assignations since the matrix $K$ is in $\GL(u,q)$, the matrix $M$ is in $\GL(m-u,q)$ and the monomial automorphism group of a $q$-ary Hamming code is transitive on the set of one-weight vectors. Hence, we have $H_{K,M,N}(\bh_i)=\lambda\bh_j$.

Secondly, consider the case when $i$ and $j$ belong to the last block of coordinate positions. Then, $\bh_i^{(u)}$ and $\bh_j^{(u)}$ are the all-zeros vector.  Now, take $N=0$, any nonsingular matrix $K$ and the matrix $M$ such that $M(\bh_i^{(m-u)})=\lambda\bh_j^{(m-u)}$. Hence, we have $H_{K,M,N}(\bh_i)=\lambda\bh_j$.

Finally,  consider the case when $i$ is in the last block of coordinate positions and $j$ is not. In this case, $\bh_i^{(u)}$ is the all-zeros vector and $\bh_j^{(u)}$ is a nonzero vector. Now, take as matrix $K$ any nonsingular matrix and the matrix $M$ such that $M(\bh_i^{(m-u)})=\lambda\bh_j^{(m-u)}$. Let $\ell$ be anyone of the nonzero coordinates of $\bh_i^{(m-u)}$ and say $\gamma$ its value. Take the matrix $N$ with all columns equal to the
all-zeros vector, except the $\ell$th column which is $\lambda\gamma^{-1}\bh_j^{(u)}$. Hence, we have $H_{K,M,N}(\bh_i)=\lambda\bh_j$. For the inverse case, when $\bh_i^{(u)}$ is a nonzero vector and $\bh_j^{(u)}$ is the all-zeros vector, we can use the same argumentation and finally take the inverse matrix of $H_{K,M,N}$.
\end{proof}

\begin{remark}
In fact, Proposition \ref{transitiu} shows that the action of $\MAut(B_{u,m})$ on the set of one-weight vectors is transitive. As a consequence, see Remark \ref{transitius}, the full automorphism group $\Aut(B_{u,m})$ is also transitive.
\end{remark}


\begin{corollary}
The code $B_{u,m}$ is completely transitive.
\end{corollary}

\begin{proof}
By \Cref{parB}, $\rho(B_{u,m})=2$. Hence, we have to see that the cosets of weight $i$ are in the same orbit, for $i=1$ and $i=2$.

Since $\Aut(B_{u,m})$ is transitive by Proposition \ref{transitiu},
we have that all the cosets of $B_{u,m}$ with minimum weight one are in the same orbit.

By Lemma \ref{autos}, it follows that $\Aut({\cal H}_u)=\GL(u,q)$ acting on the first $n_u$ coordinates is contained in $\Aut({\cal H}_m)=\GL(m,q)$, acting on the full set of $n_m$ coordinate positions. Let $B_{u,m}+\bx$ and $B_{u,m}+\by$ be two cosets of minimum weight 2, where we assume that $\bx$ and $\by$ have weight two. Let ${\cal H}_m +e_i$ and ${\cal H}_m + e_j$, with $i,j\in\{1,\ldots,n_u\}$, be the corresponding cosets of ${\cal H}_m$, according to Lemma \ref{cosets}. Since $\Aut({\cal H}_u)=\GL(u,q)$ is transitive, and by Lemma \ref{autos}, there is an automorphism $\gamma\in\Aut({\cal H}_m)$ fixing setwise the first $n_u$ coordinates (and the last $n_m-n_u$) such that $\gamma({\cal H}_m +e_i)={\cal H}_m + e_j$. By Lemma \ref{autos}, it is clear that the action of $\gamma$ in the last $n_m-n_u$ coordinates sends $B_{u,m}+\bx$ to $B_{u,m}+\by$. Indeed, if $\gamma(B_{u,m}+\bx)=B_{u,m}+\bz$, for some $\bz$ of weight two, then $e_j+\by$ and $e_j+\bz$ are codewords in ${\cal H}_m$. Thus, $\by$ and $\bz$ are in the same coset. Therefore, all the cosets of $B_{u,m}$ of weight two are in the same orbit.
\end{proof}

\subsection{The case $\rho(A)=2$}

For this case, we have the following result.

\begin{theorem}\label{rho2}
If the code $A$ has dimension $n_A-m$ and is completely regular with $\rho(A)=2$, then the supplementary code $B$, of length $n_B$, is completely regular with $\rho(B)\leq 2$.
\end{theorem}

\begin{proof}
If $A$ is completely regular with $\rho(A)=2$ then, by \Cref{params}, the external distance of $A$ is $s(A)=2$. Hence, $A^\perp$ has two nonzero weights, say $w_1$ and $w_2$. Consider any nonzero vector $\bz=(\bx\mid\by)\in {\cal H}^\perp_m$, where $\bx\in A^\perp$ and $\by\in B^\perp$. Since $\bz$ is a nonzero codeword of the simplex code of length $n_m$, we know that the weight of $\bz$ is $\w(\bz)=q^{m-1}$. Also, $\w(\bz)=\w(\bx)+\w(\by)$ and thus we obtain that $\w(\by)=q^{m-1}-w_1$ or $\w(\by)=q^{m-1}-w_2$. Note that $\bx$ cannot be the zero vector because the dimension of $A^\perp$ is $m$. We conclude that $B^\perp$ has at most two nonzero weights (if $w_1$ or $w_2$ equals $q^{m-1}$, then $B^\perp$ has only one nonzero weight). Therefore $s(B)\leq 2$, implying $\rho(B)\leq 2$, by \Cref{params}.

If $s(B)=1$, then $B$ is the trivial code of length 1, $B=\{(0)\}$, or $B$ is a Hamming code, by \Cref{params}. In any case, $B$ is completely regular. In fact, if $s(B)=1$, we are in the situation of \Cref{subseccio1}, interchanging the roles of $A$ and $B$.

If $s(B)=2$ and $\rho(B)=2$, then $B$ is a quasi-perfect uniformly packed code, by \Cref{params}. Therefore, $B$ is completely regular by \Cref{UPCR}.

Finally, note that $s(B)=2$ and $\rho(B)=1$ is not possible:
\begin{itemize}
\item[(i)] If $n_B=1$, then $s(B)$ cannot be $2$.
\item[(ii)] If $n_B=2$, then $B=\{(0,0)\}$, which has $\rho(B)=2$.
\item[(iii)] If $n_B\geq 3$, then $B$ has packing radius $e\geq 1$. Since $e\leq \rho(B)$, if we assume $\rho(B)=1$, then we have $e=\rho(B)< s(B)$ contradicting \Cref{params}.
\end{itemize}

\end{proof}

\begin{remark}
If the length of $A$ verifies $n_A>n_{m-1}$, then the zero vector cannot be a row of the parity-check matrix of $A$, otherwise $H_m$ would have two linearly dependent columns. Hence, the zero vector could not be generated by the rows of the parity-check matrix of $A$ and, as a consequence, the dimension of $A^\perp$ would be $m$. Therefore, the condition $n_A>n_{m-1}$ implies that the dimension of $A$ is $n_A-m$. Note that the converse statement is not true (see the next example).
\end{remark}

\begin{example}\label{GolayTernari}
Let $A$ be the ternary Golay $[11,6,5;2]_3$ code. Consider the ternary matrix $H_5$, which is the parity-check matrix of a ternary Hamming $[121,116,3;1]_3$ code. Let $B$ be the supplementary code which has length $n_B=110$.

Since $A$ is perfect (so completely regular) with covering radius $\rho(A)=2$, we have that $B$ is a completely regular code. Clearly, $B$ is not perfect, thus $\rho(B)=2$. Therefore, the parameters of $B$ are $[110,105,3;2]_3$. Moreover, we have computationally verified that $B$ is completely transitive and with intersection array
$$
\IA=\{220,20;1,200\}.
$$
\end{example}

Note that the hypothesis about the dimension of $A$ in \Cref{rho2} cannot be relaxed, as the next example shows.

\begin{example}\label{GolayTernari2}
Let $A$ be the punctured ternary Golay $[10,6,4;2]_3$ code. As in \Cref{GolayTernari}, consider $H_5$, the parity-check matrix of a ternary Hamming $[121,116,3;1]_3$ code. Now, let $B$ be the supplementary code with length $n_B=111$. In this case,  the dimension of $A$ is $6\neq n_A-m=5$. 

The code $A$ is completely regular and completely transitive with intersection array
$$
\IA=\{20,18;1,6\}.
$$
The code $B$ has parameters $[111,106,3;2]_3$ and it is not completely regular since its external distance is $s(B)=4$.
\end{example}

\begin{remark}
The construction described in \Cref{rho2} does not work for covering radius $\rho(A)=3$. For example, let $A$ be the extended ternary Golay code and consider the ternary matrix $H_6$, which is the parity-check matrix of a ternary Hamming $[364,258,3;1]_3$ code. Let $B$ be the supplementary code.

The code $A$ is completely transitive with $\rho(A)=3$. The code $B$ has parameters $[352,346,3;2]_3$ and it is not completely regular since its external distance is $s(B)=3$.
\end{remark}

We also give the expressions of the intersection numbers of $A$ and $B$ in terms of the lengths $n_A$ and $n_B$ and the parameter $b_1$.

\begin{corollary}\label{IAs}
Assume that the code $A$ is completely regular with dimension $n_A-m$, covering radius $\rho(A)=2$, and the supplementary code $B$ has also covering radius $\rho(B)=2$.
\begin{itemize}
    \item[(i)] The code $B$ is completely regular with dimension $n_B-m$.
    \item[(ii)] The code $A$ has intersection array
    $$
    \IA(A)=\{b_0,b_1;c_1,c_2\}=\{(q-1)n_A,b_1;1,\frac{n_A}{n_B}b_1\}.
    $$
    \item[(iii)] The code $B$ has intersection array
    $$
    \IA(B)=\{b_0',b_1';c_1',c_2'\}=\{(q-1)n_B,(q-1)n_A-\frac{n_A}{n_B}b_1;1,(q-1)n_B-b_1\}.
    $$
\end{itemize}
\end{corollary}

\begin{proof}
For (i), we already know that $B$ is completely regular, by \Cref{rho2}. Assume that the dimension of $B$ is less than $n_B-m$. Hence, the parity-check matrix of $B$ can be written containing at least one zero row. Since $s(B)=\rho(B)=2$, the dual code $B^\perp$ contains two nonzero weights, say $w_1$ and $w_2$. But, in this case, $A^\perp$ would contain three nonzero weights: $q^{m-1}$, $q^{m-1}-w_1$ and $q^{m-1}-w_2$; leading to a contradiction because $A$ has external distance $s(A)=2$.

For (ii) and (iii), with similar computations as in \Cref{comptes}, we have:
\begin{eqnarray*}
|A(1)| &=& (q-1)n_A|A|,\\
|A(2)| &=& q^{n_A}-(q-1)n_A|A|-|A|.
\end{eqnarray*}
By \Cref{graph}, we obtain
$$
b_1(q-1)n_A|A|=c_2\left(q^{n_A}-\left((q-1)n_A+1\right)|A|\right).
$$
Taking into account that $|A|=q^{n_A-m}$ and $n_B=n_m-n_A$, the expression simplifies to
\begin{equation}\label{primera}
b_1n_A=c_2n_B.
\end{equation}
By (i), we have that $|B|=q^{n_B-m}$, thus we symmetrically obtain
\begin{equation}\label{segona}
b_1'n_B=c_2'n_A.
\end{equation}

Let $J_A$ (respectively $J_B$) be the set of $n_A$ (resp. $n_B$) coordinate positions corresponding to the code $A$ (resp. $B$). Define $X_A$ (resp, $X_B$) as the set of one-weight vectors with coordinates in $J_A$ (resp. $J_B$). Define also $Y_A$ (resp. $Y_B)$ as the set of two-weight vectors in $A(2)$ (resp. $B(2)$) with coordinates in $J_A$ (resp. $J_B$). Consider the bipartite graph $\Gamma_A$ (resp. $\Gamma_B$) with vertex set $X_A\cup Y_A$ (resp. $X_B\cup Y_B$) and edges joining pairs of vertices $\bx$, $\by$, where $\bx\in X_A$, $\by\in Y_A$ (resp. $\bx\in X_B$, $\by\in Y_B$), if $d(\bx,\by)=1$.

The degree of any vertex in $X_A$ (resp. $X_B$) is $b_1$ (resp. $b'_1)$ by definition. Hence, the total number of edges in $\Gamma_A$ (resp. $\Gamma_B$) is $b_1(q-1)n_A$ (resp. $b'_1(q-1)n_B$).

Consider the set of all two-weight vectors with one nonzero coordinate in $J_A$ and one nonzero coordinate in $J_B$. There are $(q-1)^2n_An_B$ such vectors. If $e_i+e_j$ is one of these vectors ($i\in J_A$, $j\in J_B$), then there is exactly one codeword $\bz\in {\cal H}_m$ of weight three which covers $e_i+e_j$, say $\bz=e_i+e_j+e_k$. If $k\in J_A$, then $e_i+e_k\in A(2)$ and $e_i+e_k$ is a neighbor of $e_i$. Else, if $k\in J_B$, then $e_j+e_k\in B(2)$ and $e_j+e_k$ is a neighbor of $e_j$. Therefore, the vector $e_i+e_j$
induces either one edge of $\Gamma_A$, or one edge of $\Gamma_B$.

We conclude that 
$$
b_1n_A(q-1)+b_1'n_B(q-1)=(q-1)^2n_An_B,
$$
which simplifies to
\begin{equation}\label{tercera}
b_1n_A+b_1'n_B=(q-1)n_An_B.
\end{equation}

The values $b_0=(q-1)n_A$, $b'_0=(q-1)n_B$, $c_1=c_1'=1$ are trivial since $A$ and $B$ have minimum distance at least three. From \Cref{primera}, we obtain $c_2=n_Ab_1/n_B$. Using \Cref{tercera}, we compute $b_1'=(q-1)n_A-n_Ab_1/n_B$, and by \Cref{segona}, $c_2'=(q-1)n_B-b_1$.
\end{proof}

Finally, we will also show that the codes $A$ and $B$, under the hypothesis of \Cref{IAs}, are both completely transitive or both are not.

\begin{proposition}\label{subgroup}
Assume that the code $A$ has dimension $n_A-m$, covering radius $\rho(A)=2$, and let $B$ be the supplementary code. Then $\Aut(A)$ is a subgroup of $\Aut(B)$.
\end{proposition}
\begin{proof}
Let $\phi\in \Aut(A)$. Let $H_A$ (respectively $H_B$) be the parity-check matrix of $A$ (resp. $B$). Note that since the dimension of $A$ is $n_A-m$, the minimum distance of $A$ is not less than three. Let $\bh_1,\ldots, \bh_s$ be a set of $s$ columns in $H_A$ such that $\sum_{i=1}^{s} \alpha_i \bh_i=0$, where $\alpha_i\in \F_q$. Hence, we are assuming that the columns  $\bh_1,\ldots, \bh_s$ are the support of a codeword in $A$. Since $\phi$ is an automorphism of $A$ we should have $\sum_{i=1}^{s} \alpha_i \phi(\bh_i)=0$ and so the action of $\phi$ is linear over the columns in $H_A$. Since the dimension of $A$ is $n_A-m$ we can extend, by linearity, the action of $\phi$ over all columns in $H_m$ obtaining $\phi^{(e)}\in \Aut({\cal H}_m)$. The projection of $\phi^{(e)}$ over the columns of $H_B$ gives $\phi^{(e)}_B\in \Aut(B)$. It is clear that if $\phi,\psi \in \Aut(A)$ with $\phi\not= \psi$, then $\phi^{(e)}_B \not= \psi^{(e)}_B$.
\end{proof}

\begin{corollary}
    If the dimension of $A$ is $n_A-m$ and the dimension of $B$ is $n_B-m$ then $\Aut(A)$ and $\Aut(B)$ are isomorphic as abstract groups.
\end{corollary}

\begin{theorem}
    If the code $A$ is completely transitive with dimension $n_A-m$ and covering radius $\rho(A)=2$, then the supplementary code $B$ is also completely transitive.
\end{theorem}
\begin{proof}
    From \Cref{rho2}, when $A$ is a completely regular code with $\rho(A)=2$, then $B$ is completely regular with $\rho(B)\leq 2$. When $\rho(B)=1$ the code $B$ is a completely transitive code, so we are interested in proving that $B$ is a completely transitive code in the case when $\rho(B)=2$. Since $A$ is completely regular, we have that the dimension of $B$ is $n_B-m$, by \Cref{IAs}. Hence, the minimum distances of $A$ and $B$ are not less than three.
    
    Take two pairs of columns in the parity-check matrix $H_B$ of the code $B$. Say  $\bh_{i1}, \bh_{i2}$ and $\bh_{j1},\bh_{j2}$. Each pair represents a  vector of weight two and we assume that both vectors are at distance two from $B$ and also they are not in the same coset (modulo the code $B$). Therefore, we obtain two different columns $\bh_i, \bh_j$ in $H_A$ (the parity-check matrix of $A$), in such a way that both triples  $\bh_{i1}, \bh_{i2}, \bh_i$ and $\bh_{j1},\bh_{j2}, \bh_j$ are the support of codewords of weight three in ${\cal H}_m$. Since $A$ is completely transitive, there exists $\phi\in \Aut(A)$ taking one of these columns to the other and, from \Cref{subgroup}, we have an automorphism in $\Aut(B)$ taking the pair  $\bh_{i1}, \bh_{i2}$ to $\bh_{j1},\bh_{j2}$. Now, to finish the proof we need to show that taking $\bh_i, \bh_j$ two different columns in $H_B$, we have an automorphism in $\Aut(B)$ which leads $\bh_i$ to $\bh_j$.  It is easy to see that there are $\frac{n-1}{2}(q-1)^2$ codewords of minimum weight in the Hamming code ${\cal H}_m$ containing in its support the coordinate position corresponding to $\bh_i$. Of those codewords, there are $(a'_1 -(q-2))(q-1)/2$ such that they are also codewords of weight three in $B$ ($a'_1$ is the corresponding parameter of the code $B$ in \Cref{de:1.1}). Also there are $b'_1(q-1)$ codewords sharing exactly two coordinates with the support of the code $B$. Hence, 
    \begin{equation}\label{eq}
    \begin{split}
        \frac{n-1}{2}(q-1)^2 &-b'_1(q-1)-\frac{(a'_1-(q-2))(q-1)}{2}=\\ &\frac{q-1}{2}\big[(n-1)(q-1)-2b'_1 -(n_B(q-1)-b'_1-1-(q-2))\big]=\\
        &\frac{q-1}{2}\big[(n_A-1)(q-1)-b'_1 +(q-1)\big]=\\
         &\frac{q-1}{2}\big[n_A(q-1)-b'_1\big]= \frac{(q-1)}{2}\frac{n_Ab_1}{n_B}=\frac{(q-1)c_2}{2}
    \end{split}
    \end{equation}
    is the number of codewords of weight three from ${\cal H}_m$ with the coordinate corresponding to $\bh_i$ in its support and the other two coordinates in the support of the code $A$. The code $B$ will be a completely transitive code when \Cref{eq} gives a number greater than or equal to one, which is obvious. This proves the statement.
\end{proof}

It is worth mentioning that, in some cases, the construction used here can be equivalent to the construction described in \cite{BRZ4}. But this is not always the case as the following examples show.

\begin{example}\label{Ex1}
Consider the binary matrix $H_6$, which the parity-check matrix of the binary Hamming code of length $63$. Take $35$ columns of $H_6$ following the procedure described in \cite{BRZ4}, as the parity-check matrix of the code $A$. The code $A$ has parameters $[35,29,3;2]_2$ and the supplementary code $B$ has parameters $[28,22,3;2]_2$. Computationally, we have verified that both codes are completely regular but not completely transitive.
\end{example}

\begin{example}\label{Ex2}
Now, consider the binomial code $A'=C^{(7,4)}$, whose parity-check matrix $H^{(7,4)}$ contains as columns all the binary vectors of length $7$ and weight $4$ (see \cite{RZ}). Note that adding all the rows of $H^{(7,4)}$ gives the zero vector, hence the dimension of $(A')^\perp$ is $6$. The code $A'$ has parameters $[35,29,3,2]_2$ and the supplementary code $B'$ has parameters $[28,22,3;2]_2$. In this case, we have computationally verified that both codes are completely regular and also completely transitive.
\end{example}

The codes $A$ and $A'$ (respectively $B$ and $B'$) in \Cref{Ex1} and \Cref{Ex2} are completely regular codes with the same parameters. However, $A$ and $A'$ (resp. $B$ and $B'$) are not equivalent since $A'$ (resp. $B'$) is completely transitive, but $A$ (resp. $B$) is not.

\section*{Acknowledgements}

This work has been partially supported by the Spanish grants TIN2016-77918-P, (AEI/FEDER, UE). The research of the third author was carried out at the IITP RAS
at the expense of the Russian Fundamental Research Foundation (project no. 19-01-00364).

\nocite{*}

\bibliographystyle{plain}

\begin{thebibliography}{10}

%

\bibitem{BZZ}
\newblock L.A. Bassalygo, G.V. Zaitsev, V.A. Zinoviev,
\newblock Uniformly packed codes.
\newblock \emph{Problems Inform. Transmiss}, \textbf{10}, (1974), 9--14.

%


\bibitem{BRZ2}
\newblock J.~Borges, J.~Rif{\`{a}} and V.A.~Zinoviev,
\newblock On completely regular codes by concatenation constructions,
\newblock in \emph{WCC2017-10th International Workshop on Coding and
  Cryptography 2017}, 2017.

\bibitem{BRZ3}
\newblock J.~Borges, J.~Rif{\`{a}} and V.A.~Zinoviev,
\newblock On completely regular codes by concatenating Hamming codes,
\newblock \emph{Advances in Mathematics of Communications}, \textbf{12} (2018), 337--349.

\bibitem{BRZ4}
\newblock J.~Borges, J.~Rif{\`{a}} and V.A.~Zinoviev,
\newblock On new infinite families of completely regular and completely transitive binary codes,
\newblock in \emph{ACCT2018-16th International Workshop on Algebraic and Combinatorial Coding Theory 2018}, 2018.



\bibitem{BCN}
\newblock A.~E. Brouwer, A.~M. Cohen and A.~Neumaier,
\newblock \emph{Distance-Regular Graphs},
\newblock Springer, 1989.

\bibitem{Cald}
\newblock R.~Calderbank and W.~Kantor,
\newblock The geometry of two-weight codes,
\newblock \emph{Bulletin of the London Mathematical Society}, \textbf{18}
  (1986), 97--122.

\bibitem{Del}
\newblock P.~Delsarte,
\newblock \emph{An algebraic approach to the association schemes of coding
  theory},
\newblock Thesis, 1973.

\bibitem{Giu}
\newblock M.~Giudici and C.~E. Praeger,
\newblock Completely transitive codes in hamming graphs,
\newblock \emph{European Journal of Combinatorics}, \textbf{20} (1999),
  647--662.

\bibitem{GvT}
\newblock J.~Goethals and H.~VanTilborg,
\newblock Uniformly packed codes,
\newblock \emph{Philips Research Reports}, \textbf{30} (1975), 9--36.

\bibitem{HP}
\newblock W.C. Huffman and V. Pless,
\newblock \emph{Fundamentals of error-correcting codes},
\newblock Cambridge University Press, 2003.


\bibitem{Koo}
\newblock J.~Koolen, D.~Krotov and B.~Martin,
\newblock Completely regular codes, \\ \emph{https://sites.google.com/site/completelyregularcodes}.

\bibitem{Lind}
\newblock K.~Lindstr{\"o}m,
\newblock All nearly perfect codes are known,
\newblock \emph{Information and Control}, \textbf{35} (1977), 40--47.

\bibitem{MacW}
\newblock F.~J. MacWilliams and N.~J.~A. Sloane,
\newblock \emph{The theory of error-correcting codes},
\newblock Elsevier, 1977.

\bibitem{Neum}
\newblock A.~Neumaier,
\newblock Completely regular codes,
\newblock \emph{Discrete mathematics}, \textbf{106} (1992), 353--360.


\bibitem{RZ}
\newblock J.~Rif{\`a} and V.~Zinoviev,
\newblock On a class of binary linear completely regular codes with arbitrary covering radius,
\newblock \emph{Discrete mathematics}, \textbf{309} (2009),
5011--5016.

\bibitem{SZZ}
\newblock N.~Semakov, V.~A. Zinoviev and G.~Zaitsev,
\newblock Uniformly packed codes,
\newblock \emph{Problemy Peredachi Informatsii}, \textbf{7} (1971), 38--50.

\bibitem{Sole}
\newblock P.~Sol{\'e},
\newblock Completely regular codes and completely transitive codes,
\newblock \emph{Discrete Mathematics}, \textbf{81} (1990), 193--201.

\bibitem{Dam}
\newblock E.~R. van Dam, J.~H. Koolen and H.~Tanaka,
\newblock Distance-regular graphs,
\newblock \emph{arXiv preprint arXiv:1410.6294}.

\bibitem{vTi}
\newblock H.~C.~A. van Tilborg,
\newblock \emph{Uniformly packed codes},
\newblock Technische Hogeschool Eindhoven, 1976.

%

\end{thebibliography}

\end{document}